\documentclass[11pt]{article}
\usepackage[margin=2cm]{geometry}
\usepackage{graphicx}
\usepackage{psfrag}
\usepackage{amssymb,latexsym}
\usepackage{amsthm}

\newcommand{\bb}{\mathbb}

\newcommand{\R}{\bb R}

\newcommand{\Z}{\bb Z}

\newcommand{\ceil}[1]{\lceil#1\rceil}

\newcommand{\sm}{\setminus}

\newcommand{\Basic}{Bubble algorithm}
\def\st{\,:\,}
\newcommand{\scalar}[1]{\langle #1\rangle}

\newtheorem{prop}{Proposition}
\newtheorem{theorem}[prop]{Theorem}

\newtheorem{lemma}[prop]{Lemma}

\newtheorem{claim}[prop]{Claim}

\newtheorem{remark}[prop]{Remark}

\begin{document}

\title{A polynomial projection-type algorithm for linear programming}
\author{L\'aszl\'o A. V\'egh \quad \quad Giacomo Zambelli\\
 \multicolumn{1}{p{.7\textwidth}}{\centering\emph{Department of Management\\
London School of Economics and Political Science}\\
\texttt{\{L.Vegh,G.Zambelli\}@lse.ac.uk}}
}
\date{}

\maketitle

\section{Introduction}
In the
{\em linear programming feasibility problem} we are given a matrix
$A\in \Z^{m\times n}$ and a vector $b\in\Z^m$, and we wish to compute
a feasible solution to the system
\begin{equation}\label{eq:main problem}\begin{array}{c}
Ax=b\\
x\geq 0
\end{array}
\end{equation}
or show that none exists.
The first practical algorithm for linear programming was the {\em simplex method}, introduced by Dantzig in 1947 \cite{Dantzig};
while efficient in practice, for most known pivoting rules the method has an exponential-time worst
case complexity. Several other algorithms were developed over the subsequent
decades, such as the {\em relaxation method}   by Agmon \cite{Agmon} and Motzikin and Shoenberg \cite{Motzin-Shoenberg}.
The first polynomial-time algorithm, the {\em ellipsoid method}, was introduced
by Khachiyan   \cite{Khachiyan}, followed a few years later by Karamarkar's first {\em interior point method}
\cite{Karmarkar}. 
In 2010 Chubanov \cite{Chubanov-binary,Chubanov LP} gave a different type of  polynomial time algorithm, inspired by the relaxation method, followed recently by a substantially simpler and improved algorithm \cite{Chubanov new}. Computational experiments of Chubanov's original algorithm, as well as a different treatment, were carried out by Basu, De Loera and Junod \cite{Basu-DeLoera-Junod}.

Here we present a polynomial time algorithm based on \cite{Chubanov-binary}. The engine behind our algorithm is the \Basic{} subroutine, which can be considered as an unfolding of the recursion in the Divide-and-Conquer algorithm described in the earlier paper of Chubanov \cite{Chubanov-binary}. Our algorithm is also related to the one in \cite{Chubanov new}; in particular, our \Basic{} is an analogue of the Basic algorithm in  \cite{Chubanov new}. However, while our \Basic{} is a variant of
the relaxation method\footnote{In the title we however use ``projection''
instead of ``relaxation'', as it seems to be a more accurate description
of this type of algorithms. The papers \cite{Agmon,Motzin-Shoenberg} describe a
more general class of algorithms, where our algorithm corresponds to the special case that is called projection method
in \cite{Motzin-Shoenberg} and orthogonal projection method in \cite{Agmon}.}, Chubanov's Basic algorithm is precisely von Neumann's algorithm (see Dantzig \cite{Dantzig 92}).

The two algorithms proceed in a
somewhat different manner. Chubanov's algorithm decides whether $Ax=0$ has a strictly
positive solution, and reduces problems of the form (\ref{eq:main problem}) via an homogenization, whereas we work directly with the form (\ref{eq:main problem}). Also, the key updating step of  the bounds on the feasibility region after an iteration of
the basic subroutine and the supporting argument substantially differs from ours.
In particular, whereas \cite{Chubanov new} divides only one of the upper bounds on the variables by exactly two, our
algorithm uses simultaneous updates of multiple components.
Another difference is that instead of repeatedly changing the
original system by a rescaling, we keep the same problem setting
during the entire algorithm and modify a certain norm instead. This
enables a clean understanding of the progress made by the algorithm.

If we denote by $L$ the encoding size of the matrix $(A,b)$, our
algorithm performs $O([n^5/\log n]L)$ arithmetic
operations. Chubanov's algorithm \cite{Chubanov new} has a better running time bound of
$O(n^4 L)$; however, note that our algorithm is still a considerable
improvement over $O(n^{18+3\epsilon}L^{12+2\epsilon})$ in the previous
version \cite{Chubanov LP}. We get a better bound $O([n/\log n] L)$ on the number of executions of the basic subroutine, as compared to $O(nL)$ in \cite{Chubanov new}; on the other hand, \cite{Chubanov new} can use an argument bounding the overall number of elementary iterations of all executions of the basic subroutine, thus achieving a better running time estimation.

\subsection{The LP algorithm}

We highlight our polynomial-time algorithm to find a feasible solution of (\ref{eq:main problem}). We denote by $P$ the feasible region of (\ref{eq:main problem}). Throughout the paper we will assume that $A$ has full-row rank. \bigskip

Let $d^1,\ldots,d^m$ be the $m$ columns of $(A,b)$ with largest Euclidean norm,  and let $\Delta=\|d^1\|\cdots\|d^m\|$. It can be easily shown that $\Delta < 2^L$ (see for example \cite[Lemma 1.3.3]{glsbook}). By Hadamard's bound, for every square submatrix $B$ of $(A,b)$,  $|\det(B)|\leq \Delta$. It follows that, for every basic feasible solution $\bar  x$ of (\ref{eq:main problem}), there exists $q\in\Z$, $1\leq q\leq \Delta$, such that, for $j=1,\ldots,n$,  $\bar x_j=p_j/q$ for some integer $p_j$, $0\leq p_j\leq \Delta$. In particular, $\bar x_j\leq \Delta$ for $j=1,\ldots,n$, and $\bar x_j\geq \Delta^{-1}$ whenever $\bar x_j>0$.


The algorithm maintains a vector $u\in\R^n$, $u>0$, such that every basic feasible solution of (\ref{eq:main problem}) is contained in the hypercube $\{x\st 0\leq x\leq u\}$. At the beginning, we set $u_i:=\Delta$, $i=1,\ldots,n$.

At every iteration, either the algorithm stops with a point in $P$,  or it determines a vector $u'\in\R^n$, $0<u'\leq u$ such that every basic feasible solution of (\ref{eq:main problem}) satisfies $x\leq u'$ and such that, for some index $p\in\{1,\ldots,n\}$, $u'_p\leq u_p/2$. For $j=1,\ldots,n$, if $u'_j\leq \Delta^{-1}$  we reduce the number of variables by setting $x_j:=0$, and removing the $j$th column of the matrix $A$;
 otherwise, we update $u_j:=u'_j$.

The entire algorithm terminates either during an iteration when a
feasible solution is found, or once the system $Ax=b$ has a unique
solution or is infeasible. If the unique solution is nonnegative, then it gives a point in $P$, otherwise the problem is infeasible.

Since at every iteration there exists some variable $x_p$ such that $u_p$ is at least halved, and since $\Delta^{-1}\leq u_j\leq \Delta$ for every variable $x_j$ that has not been set to $0$, it follows that the algorithm terminates after at most $n\log(\Delta^2)\in O(nL)$ iterations.
%
The crux of the algorithm is the following theorem and the subsequent claim.
\begin{theorem} \label{thm:basic algorithm} There exists a strongly polynomial time algorithm which, given  $A\in\Z^{m\times n}$, $b\in\Z^m$, and $u\in\R^n$, $u>0$, in $O(n^4)$ arithmetic operations returns one of the following:
\begin{enumerate}
\item A feasible solution to (\ref{eq:main problem});
\item A vector $(v,w)\in\R^m\times\R^n_+$, $w\neq 0$, such that  $(v^\top A+w^\top)\bar x< v^\top b+\frac{1}{2n}w^\top u$ for every  $\bar x\in \{x\in\R^n\st 0\leq x\leq u\}$.
\end{enumerate}
\end{theorem}

\noindent The algorithm referred to in Theorem \ref{thm:basic algorithm} will be
called the {\em \Basic{}}, described in Section \ref{sec:basic algorithm}.

\begin{claim}\label{claim:u'} Let $u\in\R^n$, $u>0$, such that every basic feasible solution for (\ref{eq:main problem}) satisfies $x\leq u$, and let $(v,w)\in\R^m\times\R^n_+$ be a vector as in point 2 of Theorem \ref{thm:basic algorithm}.
For $j=1,\ldots,n$, let $u'_j:=\min\left\{u_j,\frac{\sum_{i=1}^n u_i w_i}{2n w_j}\right\}$. Then every basic feasible solution of (\ref{eq:main problem}) satisfies $x\leq u'$. Furthermore, if we let $p:=\arg\max_{j=1,\ldots,n}\{u_j w_j\}$, then $u'_p\leq u_p/2$.
\end{claim}
\begin{proof} Since $w\neq 0$, up to re-scaling $(v,w)$ we may assume that $\sum_{i=1}^n u_i w_i=2n$, therefore $u'_j =\min\{u_j,w^{-1}_j\}$, $j=1,\ldots,n$.  Since every basic feasible solution $\bar x$ for (\ref{eq:main problem}) satisfies $(v^\top A+w^\top)\bar x< v^\top b+\frac1{2n}w^\top u$, it follows that, for $j=1,\ldots,n$,
$$
0>v^\top (A\bar x-b)+\sum_{i=1}^nw_i\left (\bar x_i-\frac{u_i}{2n}\right)=\sum_{i=1}^n w_i\bar x_i -1
\geq  w_j \bar x_j -1.
$$
It follows that $\bar x_j< w_j^{-1}$,  thus $\bar x_j\leq u'_j$. Finally, by our choice of $p$,  $u_p w_p\geq 2 $, therefore $w_p^{-1}\leq {u_p}/2$.
\end{proof}

Theorem~\ref{thm:basic algorithm} and Claim~\ref{claim:u'} imply that
our algorithm runs in time $O(n^5L)$. In Section \ref{sec:running time
  refinement} we will refine our analysis and show that the number of
calls to the \Basic{} is actually  $O([n/\log n]L)$. This gives an
overall running time of $O([n^5/\log n]L)$.

\subsection{Scalar products}

We recall a few facts about scalar products that will be needed in the remainder.
Given a symmetric positive definite matrix $D$, we denote by $\scalar{x,y}_D=x^\top D y$. We let $\|\cdot\|_D$ the norm defined by $\|x\|_D=\sqrt{\scalar{x,x}_D}$, and refer to it as the {\em $D$-norm}. The {\em $D$-distance} between two points $x$ and $y$ is $\|x-y\|_D$. Given a point $c\in \R^n$ and  $r>0$, we define $B_D(c,r):=\{x\st \|x-c\|_D\leq r\}$, and refer to it as the {\em $D$-ball of radius $r$ centered at $c$}.

Given any system $Cx=d$ of inequalities, we denote $\langle Cx=d\rangle:=\{x\in \R^n\st Cx=d\}$. We recall that, given a point $\bar x$, assuming w.l.o.g. that $C$ has full row rank,  the point  $y$ in $\langle Cx=d\rangle$ at minimum $D$-distance from $\bar x$ is given by the formula
\begin{equation}\label{eq:projection}y=\bar x+D^{-1}C^\top (CD^{-1}C^\top)^{-1}(d-C\bar x),\end{equation}
and thus the $D$-distance between $\bar x$ and $\langle Cx=d\rangle$ is
\begin{equation}\label{eq:distance}\|y-\bar x\|_D=\sqrt{(d-C\bar x)^\top (CD^{-1}C^\top)^{-1}(d-C\bar x)}.\end{equation}

\begin{remark}\label{rmk:lambda}  If $y$ is the point in  $\langle Cx=d\rangle$ at minimum $D$-distance from $\bar x$, then the vector $\lambda:=(CD^{-1}C^\top)^{-1}(d-C\bar x)$ is the unique solution to  $y-\bar x=D^{-1} C^\top \lambda$ and $\|y-\bar x\|^2_D=(d-C\bar x)^\top \lambda$.\end{remark}

In particular, given $\alpha\in\R^n\sm\{ 0\}$ and $\beta\in \R$, and denoting by $y$ the point in $\scalar{\alpha^\top x=\beta}$ at minimum $D$-distance from $\bar x$, we have
\begin{equation}\label{eq:distance to hyperplane}
y=\bar x+\frac{D^{-1}\alpha (\beta -\alpha^\top \bar x)}{\alpha^\top D^{-1}\alpha}\;,\;
\|y-\bar x\|_D=\frac{|\beta-\alpha^\top \bar x|}{\sqrt{\alpha^\top D^{-1}\alpha}}.\end{equation}

We recall the following fact.
\begin{lemma}\label{lemma:convex-separation} Let $K\subseteq\R^n$ be a convex set, let $\bar x\in \R^n\sm K$, and let $z$ be the point in $K$ at minimum $D$-distance from $\bar x$. Then $\langle z-\bar x,x\rangle_D\geq \|z-\bar x\|^2_D$ for all $x\in K$, and $\langle z-\bar x,x\rangle_D\leq\|z-\bar x\|^2_D$ for all $x\in B_D(\bar x,\|z-\bar x\|_D)$.
\end{lemma}

Throughout the paper we denote by $e^i$ the $i$th unit vector, where the dimension of the space will be clear from the context.

\section{The \Basic{}}\label{sec:basic algorithm}

Given $u\in\R^n$ such that $u>0$, define the vector $\ell:=\frac {u}{2n}$ and let $\tilde P:=\{x \st Ax = b,\, x\geq \ell\}$.
Denote by $D=(d_{ij})$ the $n\times n$ diagonal matrix whose $i$th diagonal element is $d_{ii}=4u_i^{-2}$.
Since $u>0$, $D$ is positive definite. Furthermore, $\{x\in\R^n\st 0\leq x\leq u\}\subset B_D(0,2\sqrt n)$.
\bigskip

Throughout the rest of the paper we denote $\mathcal A:=\scalar{Ax=b}$. The \Basic{} is based on the following lemma; the proof is illustrated in Figure~\ref{fig:increase ray}.

\begin{lemma}\label{lemma:p(z)}
Let $z\in\mathcal A$ such that $\scalar{z, x}_D\geq \|z\|_D^2$ is valid for $\tilde P $.  If $z\notin P$, let $i\in\{1,\ldots,n\}$ such that $z_i<0$, and let $K:=\{x\in\mathcal A\st \scalar{z,x}_D\geq \|z\|^2_D,\, x_i\geq \ell_i\}$. Assume that $K\neq\emptyset$, and let $z'$ be the point in $K$ at minimum $D$-distance from the origin. Then $\scalar{z', x}_D\geq \|z'\|^2_D$ is valid for $\tilde P $ and  $\|z'\|^2_D> \|z\|^2_D+\frac 1{n^2}$.\end{lemma}
\begin{proof}
Since $K$  is a polyhedron containing $\tilde P $, it follows from Lemma \ref{lemma:convex-separation} that $\scalar{z', x}_D\geq \|z'\|^2_D$ is valid for $\tilde P$.
Applying (\ref{eq:distance to hyperplane}) with $\alpha =e^i$, $\beta=\ell_i$, and $\bar x=z$, we obtain that the $D$-distance between $z$ and $\scalar{x_i=\ell_i}$ is equal to
$$\frac{|\ell_i-z_i|}{\sqrt{(e^i)^\top D^{-1}e^i}}>2\frac{\ell_i}{u_i}=\frac 1 n.$$
It follows that every point in  $B_D(z,\frac 1n)$ violates the inequality $x_i\geq \ell_i$.   Next, we  show that, for  every $x\in B_D(0,\sqrt{\|z\|^2_D+\frac 1{n^2}})$, if $x$ satisfies $\scalar{z,x}_D\geq \|z\|^2_D$ then  $x\in B_D(z,\frac 1n)$. Indeed, any such point $x$ satisfies $\|z-x\|_D^2=\|z \|_D^2+\|x \|^2_D-2\scalar{z,x }_D\leq  2\|z \|_D^2+\frac 1{n^2}-2\|z\|^2_D=\frac 1{n^2}$.
Since every point in  $B_D(z,\frac 1n)$ violates $x_i\geq \ell_i$, it follows that
$B_D(0,\sqrt{\|z\|^2_D+\frac 1{n^2}})$ is disjoint from $K$ thus, by definition of $z'$, $\|z'\|^2_D> {\|z\|^2_D+\frac 1{n^2}}$.
\end{proof}

\begin{figure}[h]
\psfrag{R}{$B_D(0,\sqrt{\|z\|_D^2+\frac 1{n^2}})$}
\psfrag{r}{$r$}
\psfrag{K}{$K$}
\psfrag{z}{$z$}
\psfrag{p}{$z'$}
\psfrag{B}{$B_D(z,\frac 1n)$}
\psfrag{0}{$0$}
\psfrag{h1}{$\scalar{z,x}_D\geq \|z\|^2_D$}
\psfrag{h2}{$x_i\geq \ell_i$}
\begin{center}
\includegraphics[scale=.6]{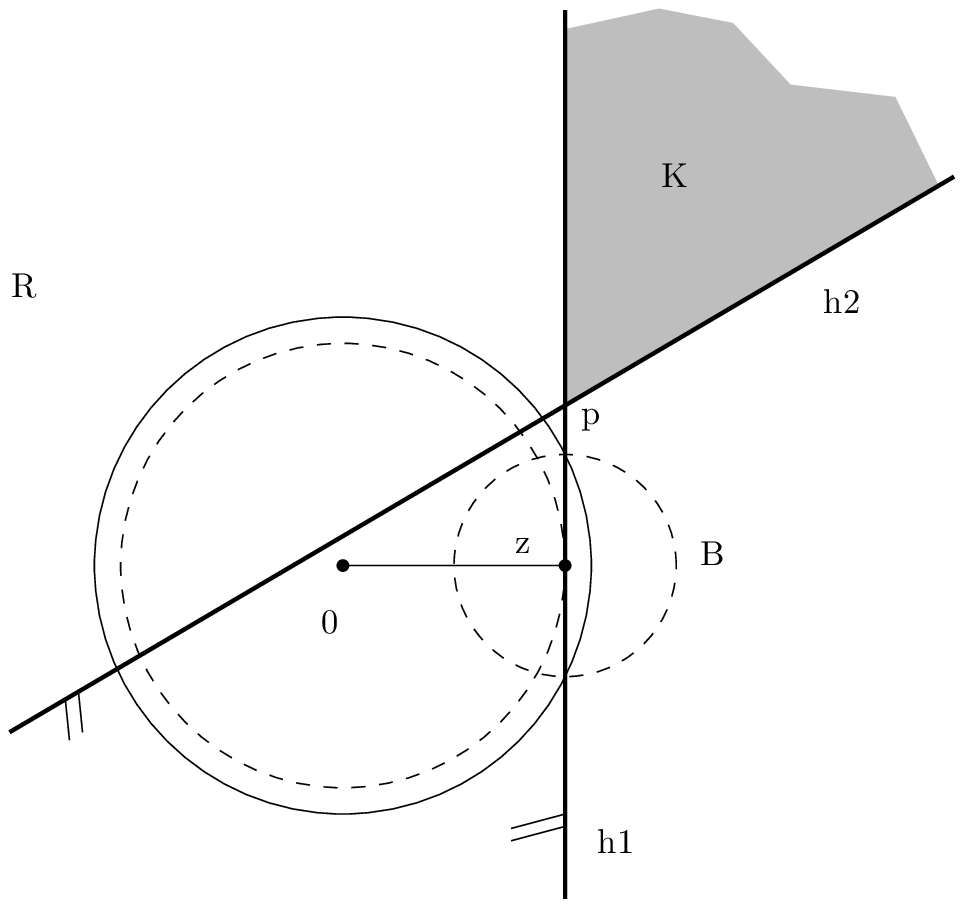}
\caption{\label{fig:increase ray} Bubbles merge: the
  $D$-ball $B_D(0,\sqrt{\|z\|_D^2+\frac 1{n^2}})$ is disjoint from $K$.}
\end{center}
\end{figure}

Note that, if $\scalar{z,x}_D\geq \|z\|_D^2$ is a valid inequality for $\tilde P$, then there exists $(v,w)\in\R^m\times\R^n_+$ such that $Dz=  A^\top v+w$ and $\|z\|^2_D=v^\top b+w^\top \ell$, because $\scalar{z,x}_D=(Dz)^\top x$.  Also, whenever $K$ in the statement of Lemma~\ref{lemma:p(z)} is empty, there exists a vector $(v,w)\in\R^m\times\R^n_+$ such that $A^\top v+w=0$ and $v^\top b+w^\top \ell>0$. We will detail in Section \ref{sec:closest point} how these vectors $(v,w)$ can be computed. Lemma \ref{lemma:p(z)} and the above considerations suggest the algorithm in Figure~\ref{fig:bubble}.

\begin{figure}[h]{\sl
\noindent{\bf \Basic{}}
\medskip

\noindent{\bf Input: } A system $Ax=b$, $x\geq 0$, and a vector $u>0$.
\smallskip

\noindent{\bf Output: } Either:

\noindent\hspace{.2cm} 1.~A point in $P$, or;

\noindent\hspace{.2cm} 2.~$(v,w)\in\R^m\times\R^n_+$ such that

\noindent\hspace{.2cm} \phantom{2.} $(v^\top A+w^\top)x< v^\top b+w^\top \ell$~~$\forall\,x\in B_D(0,2\sqrt{n})$.
\bigskip

\noindent Initialize $z$ as the point in $\mathcal A$ at minimum $D$-distance
from $0$.
Set $\ell:=\frac {u}{2n}$.

\noindent{\bf While} $\|z\|_D\leq 2\sqrt{n}$, do

\noindent\hspace{.2cm}{\bf If} $z\in P$, STOP;

\noindent\hspace{.2cm }{\bf Else}, Choose $i$ such that  $z_i<0$;

\noindent\hspace{.4cm}Let $K:=\{x\in\mathcal A\st \scalar{z,x}_D\geq \|z\|^2_D,\, x_i\geq \ell_i\}$.

\noindent\hspace{.4cm}{\bf If} $K=\emptyset$, {\bf then} output $(v,w)\in\R^m\times\R^n_+$

\noindent\hspace{.6cm}such that $A^\top v+w=0$ and $v^\top b+v^\top \ell>0$;

\noindent\hspace{.4cm}{\bf Else} Reset $z$ to be the point in $K$

\noindent\hspace{.6cm} at minimum $D$-distance from $0$,

\noindent{\bf Endwhile;}
\smallskip

\noindent Output $(v,w)\in\R^m\times\R^n_+$ such that\\ $Dz=  A^\top v+w$ and $\|z\|^2_D=v^\top b+w^\top \ell$.
}
\caption{The Bubble algorithm}\label{fig:bubble}
\end{figure}

By Lemma~\ref{lemma:p(z)}, the value of $\|z\|_D^2$ increases by at least $\frac 1{n^2}$ at every iteration. Therefore, after at most $4n^3$ iterations, $\|z\|_D>2\sqrt{n}$. In particular, if the algorithm terminates outside the ``while'' cycle, then the inequality $\scalar{z,x}_D\geq \|z\|^2_D$ is valid for $\tilde P$ and it is violated by every point in $B_D(0, {2\sqrt n})$, and therefore by every point in $\{x\st 0\leq x\leq u\}$. Note that this is the second outcome of Theorem \ref{thm:basic algorithm}.

To show that the \Basic{} is strongly polynomial, we need to address two issues. The first is how to compute, at any iteration, the new point $z$ and a vector $(v,w)\in\R^m\times \R^n_+$ such that $Dz=  A^\top v+w$ and $\|z\|^2_D=v^\top b+w^\top \ell$, and how to compute a Farkas certificate of infeasibility for $\tilde P$ if the \Basic{} stops with $K=\emptyset$.
In Section~\ref{sec:closest point}, we show how this step can be implemented in $O(n)$ time, and therefore \Basic{} performs $O(n^4)$ arithmetic operations.
 The second issue is that, in the above algorithm, the encoding size of the vector $z$ computed could grow exponentially. In Section~\ref{sec:numerical} we show how this can be avoided by performing an appropriate rounding at every iteration.

Note that our \Basic{} terminates with either a point in $P$, or with a separating hyperplane. The latter arises in two cases: either if $K=\emptyset$ in a certain iteration, or at the end of the while cycle.
The Divide-and-Conquer algorithm in Chubanov \cite{Chubanov-binary} and Basu et al. \cite{Basu-DeLoera-Junod} has three possible outcomes, including a ``failure'' scenario, corresponding to $K=\emptyset$. Their reason for handling this scenario separately is due to an initial homogenization step; we do not have to make such a distinction as we do not perform such an homogenization.

\subsection{Computing  the closest point in $K$}
\label{sec:closest point}

Instead of maintaining $(v,w)\in\R^m\times \R^n_+$ with $Dz=  A^\top v+w$ and $\|z\|^2_D=v^\top b+w^\top \ell$, it is more convenient to work with a different representation inside the affine subspace $\mathcal A$, as detailed below.

Let us denote by $r^0$ the point in $\mathcal A$ at minimum $D$-distance from the origin. By (\ref{eq:projection}) and Remark \ref{rmk:lambda}, using Gaussian elimination we can compute in strongly polynomial time $\bar v^0\in\R^m$ such that $Dr^0=A^\top \bar v^0$ and $\|r^0\|^2_D=b^\top \bar v^0$.

\begin{remark}\label{rmk:distance from origin} Observe that, for every $x\in \mathcal A$, $\scalar{x-r^0, r^0}_D=0$, thus $\|x\|^2_D=\|x-r^0\|_D^2+\|r^0\|_D^2$. It follows that,  for any convex set $C\subseteq\mathcal A$, the point in $C$ at minimum $D$-distance from the origin is the point in $C$ at minimum $D$-distance from $r^0$.
\end{remark}

For $j=1,\ldots,n$, we may assume that $\emptyset \neq \{x\in\mathcal A\st x_j\geq \ell_j\}\subsetneq \mathcal A$. Under this assumption, there exists  $\alpha^j\in\mathcal A-r^0$ and $\beta_j\in\R$ such that  $\|\alpha^j\|_D=1$ and $\{x\in\mathcal A\st x_j\geq\ell_j\}=\{x\in\mathcal A\st \scalar{\alpha^j,x}_D\geq\beta_j\}$. Note that $(\alpha^j,\beta^j)$ can be computed using Gaussian elimination, along with $\bar v^j\in\R^m$ and $\bar w_j\in\R_+$ such that $D\alpha^j=A^\top \bar v^j+\bar w_j e^j$,  $\beta_j=b^\top \bar v^j+\bar w_j \ell_j$. Observe that, if we denote by $r^j$ the point in $\{x\in\mathcal A\st x_j=\ell_j\}$ at minimum $D$-distance from the origin, then $|\beta|=\|r^j-r^0\|_D$ and $r^j=r^0+\beta \alpha_j$.

We may assume that $r^0\notin P$, otherwise the algorithm terminates immediately. It follows that at the first iteration $z=r^t$ for some $t$ such that $\beta_t>0$. We can assume that at the first iteration, we choose $t=\arg\max_{j}\beta_j$. In particular, $\|z\|_D\geq \|r^j\|_D$ for all $j$ such that $\beta_j>0$. By Lemma \ref{lemma:p(z)} and Remark \ref{rmk:distance from origin}, $\|z-r^0\|_D\geq \frac 1n$.

\bigskip

At any subsequent iteration, we are given a point  $z$ in $\mathcal A\sm P$ such that $\scalar{z,x}_D\geq \|z\|_D^2$ is valid for $\tilde P$. Let $\alpha=(z-r^0)/\|z-r^0\|_D$ and $\beta=\|z-r^0\|$. It follows that $\{x\in \mathcal A\st \scalar{z,x}_D\geq \|z\|_D^2\} =\{x\in \mathcal A\st \scalar{\alpha,x}_D\geq \beta\}$. We will maintain a vector $\lambda\in\R^n_+$ such that
\begin{equation}
\label{def:lambda}
(\alpha,\beta)=\sum_{j=1}^n \lambda_j (\alpha^j,\beta_j).
\end{equation}
These provide the vectors $(v,w)$ with $Dz=A^\top v+w$,
$\|z\|_D^2=b^\top v+\ell^\top w$, and $w\geq 0$, by
defining
 $v:=\bar v^0+\beta\sum_{j=1}^n \lambda_j \bar v^j$ and $w:=\beta \sum_{j=1}^n\lambda_j \bar w_j e^j$.

\bigskip

In every iteration,  our algorithm terminates if $z\ge 0$, or it picks an index with
$z_i<0$ and defines $K:=\{x\in\mathcal A\st \scalar{z,x}_D\geq \|z\|^2_D,\, x_i\geq \ell_i\}$.

If $K=\emptyset$ then the algorithm terminates; otherwise, the current $z$ is replaced by the point $z'$ in $K$ at minimum $D$-distance from the origin. 

In the rest of this section, we describe how to compute $z'$ if $K\neq \emptyset$, or a Farkas certificate of infeasibility if $K=\emptyset$.
Let
$$
\bar K:=\{x\in\R^n\st \scalar{\alpha,x}_D\geq\beta,\, \scalar{\alpha^i,x}_D\geq\beta_i\}.
$$
It follows that $K=\mathcal A\cap \bar K$.

\begin{claim}\label{claim:K and K bar}  If $\bar K \neq \emptyset$, then the point in $\bar K$ at minimum $D$-distance from $r^0$ is equal to $z'$. In particular $K\neq\emptyset$ if and only if $\bar K\neq\emptyset$.
\end{claim}
\begin{proof} Let  $\bar z$ be the point in $\bar K$ at minimum $D$-distance from $r^0$. Since $\scalar{\bar z-r^0,x}_D\geq \|\bar z-r^0\|^2_D$ is valid for $\bar K$, it follows that $D(\bar z-r^0)=\mu_1 D\alpha^i+\mu_2 D\alpha$ for some $\mu_1,\mu_2\geq 0$. From this, we get $\bar z=r^0+\mu_1 \alpha^i+\mu_2\alpha$, which implies $A\bar z=Ar^0+\mu_1 A\alpha^i+\mu_2A\alpha=b$. This shows that $\bar z\in \mathcal A$, and thus $\bar z\in K$. Since $K\subseteq \bar K$, it follows that $\bar z$ is the point in $K$ at minimum $D$-distance from $r^0$, and thus from the origin, i.e. $\bar z=z'$. In particular, if $\bar K\neq\emptyset$ then also $K\neq \emptyset$. Conversely, if $K\neq \emptyset$ then $\bar K\neq \emptyset$ because $K\subseteq \bar K$.
\end{proof}

\begin{claim}\label{claim:K nonempty} $K \neq \emptyset$ if and only if $\alpha^i$ and $\alpha$ are linearly independent. If $K\neq\emptyset$, then $z'$ is the point in
$$
\mathcal L:=\scalar{\scalar{\alpha,x}_D=\beta,\, \scalar{\alpha^i,x}_D=\beta_i}
$$
 at minimum $D$-distance from $r^0$.
\end{claim}
\begin{proof} Assume that $ K\neq \emptyset$. By Claim \ref{claim:K and K bar}, $z'$ is the point in $\bar K$ at minimum $D$-distance from $r^0$.  We will show that $z'$ is the point in $\mathcal L$ at minimum $D$-distance from $r^0$. It suffices to show that $z'$ satisfies $\scalar{\alpha,z'}_D=\beta$ and $\scalar{\alpha^i,z'}_D=\beta_i$. If $\scalar{\alpha^i,z'}_D>\beta^i$, then $z'$ is the point in $\{x\st \scalar{\alpha,x}_D\geq\beta\}$ at minimum $D$-distance from $r^0$, and thus $z'=z$, contradicting the fact that $\|z'\|_D>\|z\|_D$. If $\scalar{\alpha,z'}_D=\beta$, then $z'$ is the point in $\{x\st  \scalar{\alpha^i,x}_D\geq  \beta_i\}$ at minimum $D$-distance from $r^0$. If $\beta_i>0$ then $z'=r^i$, contradicting the fact that $\|z'\|_D>\|z\|_D\geq \max_{j}\beta_j$, whereas if $\beta_i\leq 0$ then $z'=r^0$, contradicting the fact that $r^0\notin \bar K$.

For the first part of the statement, by Claim \ref{claim:K and K bar}, $K\neq\emptyset$ if and only if $\bar K\neq\emptyset$. Clearly $\bar K\neq \emptyset$ if $\alpha^i$ and $\alpha$ are linearly independent. Conversely, assume $\bar K\neq \emptyset$. If $\alpha^i$ and $\alpha$ are linearly dependent, then, because  $z'$ is the point in $\mathcal L$ at minimum $D$-distance from $r^0$, it follows that $\mathcal L=\scalar{\scalar{\alpha^i,x}_D=\beta^i}=\scalar{\scalar{\alpha,x}_D=\beta}$ and that $z'=z$, a contradiction.
\end{proof}

\noindent{\sl Case $K\neq \emptyset$.} By Claim~\ref{claim:K nonempty} $z'$ is the closest point in $\mathcal L$ from $r^0$, and $\alpha^i$ and $\alpha$ are linearly independent. According to Remark \ref{rmk:lambda},  we have that $z'-r^0=\mu^1 \alpha +\mu^2 \alpha^i$, where $(\mu_1,\mu_2)^\top
=(CD^{-1}C^\top)^{-1}(d-Cr^0)$,  $C$ being the $2\times n$ matrix
whose rows are $(D\alpha^i)^\top$ and $D\alpha^\top$, and where $d\in\R^2$ is defined by
$d_1=\beta_i$, $d_2=\beta$.
A simple computation gives that
\begin{equation}\label{eq:mu}
\mu_1=\frac{\beta^i-\beta\scalar{\alpha^i,\alpha}_D}{1-\scalar{\alpha^i,\alpha}_D^2}\quad
\mu_2=\frac{\beta-\beta^i\scalar{\alpha^i,\alpha}_D}{1-\scalar{\alpha^i,\alpha}_D^2}.
\end{equation}
We also claim that  $\mu_1,\mu_2\geq 0$. Indeed, $\scalar{z'-r^0,x}_D\geq \|z'-r^0\|_D^2$ is a valid linear inequality for $\bar K$, and $\mu_1$ and $\mu_2$ are the unique coefficients satisfying  $D(z'-r_0)=\mu_1\alpha^i+\mu_2\alpha$ and $\|z'-r^0\|_D^2=\mu_1\beta^i+\mu_2\beta$.
Defining
\begin{equation}
\beta'=\|z'-r^0\|_D,\, \alpha'=(z'-r^0)/\beta',\, \lambda':=(\mu_1 e^i+\mu_2\lambda)/\beta',\label{def:lambda'}
\end{equation}
we have that  $\lambda'\geq 0$ and $(\alpha',\beta')=\sum_{j=1}^n \lambda'_j (\alpha^j,\beta_j)$.
Therefore, $z'$ and $\lambda'$ can be computed by performing $O(n)$ arithmetic operations at every iteration of the \Basic{}.
\begin{remark} \label{rmk:postive lambdas} Since $z'\in\scalar{\scalar{\alpha^i,x}_D= \beta_i}$, it follows that $\|z'\|_D\geq \|r^i\|_D$. Therefore, at every iteration of the algorithm, $|\beta_j|\leq \beta$ whenever $\lambda_j>0$.
\end{remark}
\bigskip

\noindent{\sl Case $K= \emptyset$.} By Claim~\ref{claim:K nonempty}, $\bar K=\emptyset$ and the vectors $\alpha_i$ and $\alpha$ are linearly dependent. This implies that, for some $\nu>0$, $\alpha^i=-\nu \alpha$  and $\beta^i>-\nu \beta$. Defining $\lambda':=e^i+\nu\lambda$, we obtain that $\sum_{j=1}^n \lambda'_j \alpha^j=0$ and $\sum_{j=1}^n\lambda'_j \beta_j>0$.

A Farkas certificate of infeasibility $(v',w')$ can be obtained by setting  $v':=\bar v^0+\sum_{j=1}^n\lambda'_j\bar v^j$ and $w':=\sum_{j=1}^n\lambda'_j\bar w_j e^j$. We thus have $A^\top v'+w'=0$, $w'\geq 0$ and  $b^\top v'+\ell^\top w'>0$, showing infeasibility of $\tilde P$.

\subsection{Bounding the encoding sizes}\label{sec:numerical}

Note that the encoding size of the vector $z$ within the \Basic{} could grow exponentially, and also the size of the upper bound vector $u$ maintained by the LP algorithm. 

To maintain the size of $u$ polynomially bounded, we will perform a rounding at every iteration as follows. Instead of the new bounds $u'_j$ as in
Claim~\ref{claim:u'}, let us define $\tilde u_j$ as the smallest
integer multiple of $1/(3n\Delta)$ with $u'_j\le \tilde u_j$. We
proceed to the next iteration of the \Basic{} with the input vector
$\tilde u$. Clearly the encoding size of $\tilde u$ is polynomially bounded in $n$ and $L$, and we shall show in the next section that this rounding does not affect the asymptotic running time bound of $O([n^5/\log n]L)$.

\medskip
In the rest of this section we show how  a  rounding step can be introduced in the \Basic{}
in order to  guarantee that the sizes of the numbers remain polynomially bounded. The rounding will be performed on the coefficients
$\lambda_j$ in (\ref{def:lambda}).

In every iteration, after the new values of $z$ and the $\lambda_j$'s are obtained, we replace them by $\tilde z$ and $\tilde \lambda_j$ satisfying (\ref{def:lambda}), such that these values have polynomial encoding size.
At the same time, we show that $\|z\|^2_D-\|\tilde z\|^2_D\leq \frac
1{2n^2}$ (Claim~\ref{cl:tilde-z}); since at every iteration of the \Basic{} the value of $\|z\|_D^2$ increases by at least $\frac 1{n^2}$, the number of iterations in \Basic{} may increase by at most a factor of 2, to $8n^3$. Let
$$
q:=\ceil{16 n^{3}}.
$$

For every number $a\in \R$, we denote by $[a]_q$ the number of the form $p/q$, $p\in\Z$, with $|p/q- a|$ minimal. Given the current point $z$ and $\lambda\in \R^n_+$ satisfying (\ref{def:lambda})
let
$$
(\gamma,\delta):=\sum_{j=1}^n [\lambda_j]_q (\alpha^j,\beta^j).
$$
It follows that $\scalar{\gamma,x}_D\geq \delta$ is a valid inequality for $\tilde P$. Let us define  $\tilde z$ as the closest point in $\scalar{\scalar{\gamma,x}_D=\delta}$ to $r^0$.
This can be obtained by
\[
\tilde \alpha:=\gamma/\|\gamma\|_D,\quad \tilde\beta =\delta/\|\gamma\|_D,\quad \tilde z:=r_0+\tilde \alpha\tilde \beta,\quad  \tilde \lambda_j:=[\lambda_j]_q/\|\gamma\|_D
\]
Note that $(\tilde \alpha,\tilde\beta)=\sum_{j=1}^n\tilde\lambda_j(\alpha^j,\beta^j)$ and
 $\|\tilde z-r^0\|_D=|\tilde \beta|$ hold. The next claim will show that $\tilde \beta>0$.

\begin{claim}\label{cl:tilde-z} $\|z\|^2_D-\|\tilde z\|^2_D\leq \frac 1{2n^2}$ and $\tilde \beta>0$.
\end{claim}
\begin{proof}
We first show that $\|\alpha-\gamma\|_D\leq \frac{n}{2q}$ and $|\beta-\delta|\leq \frac{n\beta}{2q}$. Indeed,
$\|\alpha-\gamma\|_D^2=\sum_{j,h=1}^n(\lambda_j-[\lambda_j]_q)(\lambda_h-[\lambda_h]_q)\scalar{\alpha^j,\alpha^h}_D\leq \frac{n^2}{4q^2}$, because $\scalar{\alpha^j,\alpha^h}_D\leq 1$ for $j,h=1,\ldots,n$. Also,
$|\beta-\delta|= |\sum_{j=1}^n (\lambda_j-[\lambda_j]_q)\beta^j|\leq \frac{n\beta}{2q}$, because by Remark \ref{rmk:postive lambdas} $|\beta^j|\le\beta$ whenever $\lambda_j>0$.  Note that $\delta\ge \beta\left(1-\frac{n}{2q}\right)>0$, thus $\tilde \beta>0$, proving the second claim.

We assume that $\|z\|_D\geq\|\tilde z\|_D$, otherwise the first claim is trivial. Note that $\|z\|_D-\|\tilde z\|_D=\beta-\frac{\delta}{\|\gamma\|_D}\leq |\beta-\delta|+\frac{\delta}{\|\gamma\|_D}|\|\gamma\|_D-1|\leq \frac{\beta n}{q}$, where the last inequality follows from $|\beta-\delta|\leq \frac{n\beta}{2q}$, $\tilde \beta\leq \beta$ and from $|\|\gamma\|_D-1|=|\|\gamma\|_D-\|\alpha\|_D|\leq \|\alpha-\gamma\|_D\leq \frac{n}{2q}$.
Finally $\|z\|^2_D-\|\tilde z\|^2_D=(\|z\|_D-\|\tilde z\|_D)(\|z\|_D+\|\tilde z\|_D)\leq \frac{2\beta\|z\|_D n}{q}\leq \frac{8n^2}{q}\leq \frac 1 {2n^2}$. The second inequality follows since $\beta\le \|z\|_D\le 2\sqrt{n}$ by the termination criterion of the \Basic{}.
\end{proof}

We claim that the encoding size of the $\tilde \lambda_j$'s remains polynomially bounded. Note that, if at every iteration we guarantee that $\lambda_j$ ($j=1,\ldots,n$) is bounded from above by a number of polynomial size, then the encoding sizes of $[\lambda_j]_q$ is polynomial as well, and therefore so is the encoding size of $\tilde \lambda$.

Let $z$ and $\lambda$ satisfy (\ref{def:lambda}), and let
$z'$ denote the next point, with $\lambda'$ defined by (\ref{def:lambda'}); note that
$z'$ and $\lambda'$ also satisfy (\ref{def:lambda}).

\begin{claim} If $\|z'\|_D\leq 2\sqrt{n}$, then  $\lambda'_j\leq 8n^3(\lambda_j+1)$ for $j=1,\ldots,n$.
\end{claim}
\begin{proof}
Let $\mu_1,\mu_2\geq 0$ be defined as in (\ref{eq:mu}); recall that
these satisfy $z'-r^0=\mu_1\alpha^i+\mu_2\alpha$ and
$\|z'-r^0\|_D^2=\mu_1\beta_i+\mu_2\beta$.
We first show that
$\mu_1,\mu_2\leq 8n^3\beta$. It follows from (\ref{eq:mu}) that
\begin{eqnarray*}
\|z'-r^0\|^2_D&=&\frac{\beta_i^2+\beta^2-2\beta\beta^i\scalar{\alpha,\alpha^i}_D}{1-\scalar{\alpha^i,\alpha}^2_D}=\\
&=&\frac{\|r^i-r^0\|^2_D+\|z-r^0\|^2_D-2\scalar{r^i-r^0,z-r^0}_D}{1-\scalar{\alpha^i,\alpha}^2_D}=\frac{\|z-r^i\|^2_D}{1-\scalar{\alpha^i,\alpha}^2_D}.
\end{eqnarray*}
In the second equality we use $r^i-r^0=\alpha^i\beta^i$,
$z-r^0=\alpha\beta$, and $\|\alpha^i\|_D=\|\alpha\|_D=1$.
Since $z$ has distance at least $1/n$ from the hyperplane
$\scalar{\scalar{r^i,x}_D=\|r^i\|^2_D}$, it follows that
$\|z-r^i\|_D\geq 1/n$. Since $\|z'-r^0\|_D\leq 2\sqrt n$, we obtain
$1-\scalar{\alpha^i,\alpha}^2_D\geq 1/4n^3$. Further, by
Remark~\ref{rmk:postive lambdas} we have $|\beta_i|\leq\beta$. These
together with  (\ref{eq:mu}) and $|\scalar{\alpha^i,\alpha}_D|\leq 1$, imply
$\mu_1,\mu_2\leq 8n^3\beta$.

Since (\ref{def:lambda'}) defines $\lambda'=(\mu_1 e^i+\mu_2\lambda)/\|z'-r^0\|$, using that $\|z'-r^0\|>\beta$, it follows that $\lambda'_j\leq 8n^3(\lambda_j+1)$ for $j=1,\ldots,n$.
\end{proof}

Since at the first iteration $\lambda_j\leq 1$, $j=1,\ldots,n$, it
follows from the above claim that after $k$ iterations of the
\Basic{}, we have  $\lambda_j\leq k(8n^3)^{k}$. 
As argued above, the rounded \Basic{} terminates in at most $8n^3$ iterations, therefore  $\lambda_j\in O(n^3 (8n^3)^{8n^3})$ in all iterations.
Consequently, the $\lambda_j$'s encoding sizes are polynomially bounded.

Since every iteration of the \Basic{}can be carried out in $O(n)$ arithmetic operations and the encoding size of the numbers remains polynomially bound it follows that  the \Basic{} is strongly polynomial, with running time $O(n^4)$.

\section{Improving the running time by a $1/\log(n)$ factor}
\label{sec:running time refinement}
In this section, we show that the total number of calls of the \Basic{} can be bounded by
 $O\left(\frac{n}{\log n}L\right)$. This will be achieved through an amortized runtime analysis by means of a potential.
For simplicity, we first present the analysis for the version where the updated vector $u$ is not rounded, and then explain the necessary modifications when rounding
is used.

A main event in the LP algorithm is when, for some coordinate $j$, we obtain $u_j<\Delta^{-1}$ and therefore we may conclude $x_j=0$.
This reduces the number of variables by one.
The algorithm terminates once the system $Ax=b$ has a unique solution
or is infeasible;
assume this happens after eliminating $f\le n-1$ variables.
For simplicity of notation, let us assume that the variables are set to zero in the order $x_n,x_{n-1},\ldots,x_{n-f+1}$, breaking ties arbitrarily.
For $k=2,\ldots,f$, the {\em $k$-th phase} of the algorithm starts with the iteration after the one when $x_{n+2-k}$ is set to zero and terminates when $x_{n+1-k}$ is set to zero; the first phase consists of all iterations until $x_n$ is set to 0.
A phase can be empty if multiple variables are set to 0 simultaneously. In the $k$-th phase there are $n-k+1$ variables in the problem.
We analyze the potential
$$
\Psi:=\sum_{j=1}^n \log\max\left\{u_j,\frac1{\Delta}\right\}.
$$
Note that the initial value of $\Psi$ is $n\log \Delta$, and it is monotone decreasing during the entire algorithm.
Let $p_k$ denote the decrease in the potential value in phase $k$. Since $u_j$ decreases from $\Delta$ to at most $\frac1{\Delta}$ for every $j=n+1-k,\ldots,n$ in the first $k$ phases, we have that the value of $\Psi$ at the end of the $k$-th phase is at most $(n-2k)\log\Delta$, or equivalently,
\begin{equation}\label{eq:sum-p}
\sum_{i=1}^k p_i\ge 2k\log\Delta.
\end{equation}

\begin{claim}\label{claim:psi-dec}
In the $k$-th phase of the algorithm, $\Psi$ decreases  by at least
 $\log(n-k+2)$ in every iteration, with the possible exception of the last one.
\end{claim}
\begin{proof} Note that in the $k$-th phase the values $u_n,u_{n-1},\ldots,u_{n-k+2}$ do not change anymore.
Recall from Claim~\ref{claim:u'} that, for $j=1,\ldots,n-k+1$, the new value of $u_j$ is set as  $u'_j:=\min\left\{u_j,w_j^{-1}\right\}$, where $(v,w)$ is a vector as in Theorem~\ref{thm:basic algorithm} normalized to $\sum_{i=1}^{n-k+1} u_iw_i=2(n-k+1)$, since $n-k+1$ is the number of variables not yet fixed to $0$.
In particular, $u_j/u'_j=\max\{1,u_jw_j\}$ for  $j=1,\ldots,n-k+1$.
The new value of the potential is $\Psi'=\sum_{j=1}^n \log\min\left\{u'_j,\frac1{\Delta}\right\}$. In every iteration of the $k$-th phase except for the last one we must have $u_j\ge u'_j>\frac1{\Delta}$ for every $j=1,\ldots,n-k+1$, hence
\begin{equation}\label{eq:potential drop}
{\Psi}-{\Psi'}= \sum_{j=1}^{n-k+1}\log\frac{u_j}{u'_j}=\log\prod_{j=1}^{n-k+1} \max\{1,u_jw_j\}\geq \log(n-k+2).
\end{equation}
The last inequality follows from the fact that, for any positive integer $t$, we have $$\min\left\{\prod_{j=1}^{t} \max\{1,\alpha_j\}\st \sum_{j=1}^{t} \alpha_j=2t,\, \alpha\in\R^t_+\right\}=t+1,$$ the minimum being achieved when $\alpha_j=t+1$ for a single value of $j$ and $\alpha_j=1$ for all other values.
\end{proof}
Let $r_k$ denote the number of iterations in phase $k$. The claim implies the following upper  bound:
\begin{equation}\label{eq:r-bound}
r_k\le \frac{p_k}{\log (n-k+2)}+1
\end{equation}

Together with (\ref{eq:sum-p}), it follows that the total number of iterations $\sum_{k=1}^f r_k$ is bounded by the  optimum of the following LP
\begin{equation}\label{eq:LP potential}\begin{array}{rl}
\max\ & (n-1)+\sum_{i=1}^{n} \frac{1}{\log (n-i+2)} p_i\\
\sum_{i=1}^k p_i&\ge 2k\log\Delta\quad  k=1,\ldots,n-1\\
\sum_{i=1}^n p_i&\le 2n\log\Delta\\
p&\in\R^{n}_+
\end{array}\end{equation}
It is straightforward that the optimum solution is $p_i=2\log\Delta$ for $i=1,\ldots,n$. One concludes that the number of iterations is at most
$$
(n-1)+2\log\Delta \sum_{i=1}^{n}  \frac{1}{\log (n-i+2)} \leq (n-1)+ 2\log\Delta\int_{1}^{n+1}\frac{dt}{\log t},$$
where the inequality holds because the function $\frac{1}{\log x}$ is decreasing.
The function  $\mathrm{li}(x):=\int_{0}^{x}\frac{dt}{\ln t}$ (defined for $x>1$) is the {\em logarithmic integral} \cite{Abramowitz-Stegun}, and it is known that $\mathrm{li}(x)=O\left(\frac{x}{\log x}\right)$. This gives the bound $O\left(\frac{n}{\log n}L\right)$ on the total number of iterations,
using that $\Delta\le 2^L$.

Let us now turn to the version of the algorithm where $u'_j$ is rounded up to
$\tilde u_j$, an integer multiple of $1/(3n\Delta)$. In all but the
last iteration of the $k$'th phase $u'_j\ge
1/\Delta$ holds, and therefore $\frac{u_j}{\tilde u_j}\ge \frac{u_j}{u_j'}\cdot
\frac{1}{1+1/(3n)}$. Hence, from (\ref{eq:potential drop}), in the $k$'th phase we have $\Psi-\Psi'\ge
\log(n-k+2)-(n-k+1)\log \frac{3n+1}{3n}> \frac12\log(n-k+2)$.
This ensures at least half of the drop in potential guaranteed in Claim~\ref{claim:psi-dec}, giving the same asymptotic running time bound.

\end{document}